\documentclass[reqno]{amsart}
\usepackage[utf8]{inputenc}
\usepackage[T1]{fontenc}
\usepackage[english]{babel}
\usepackage{mathtools}
\usepackage{amssymb}
\usepackage{amsthm}
\usepackage{MnSymbol}
\usepackage{bbm}
\usepackage{amssymb,fge}
\usepackage{enumitem}

\mathtoolsset{showonlyrefs,showmanualtags}

\newtheorem{Satz}{Satz}[section]
\newtheorem{Lem}[Satz]{Lemma}
\newtheorem{Ex}[Satz]{Example}
\newtheorem{Def}[Satz]{Definition}
\newtheorem{Cor}[Satz]{Corollary}
\newtheorem{Pro}[Satz]{Proposition}
\newtheorem{Thm}[Satz]{Theorem}

\newtheorem{Rem}[Satz]{Remark}

\newcommand{\R}{\mathbb{R}}

\newcommand{\inv}{^{-1}}
\newcommand{\Leb}{\mathcal{L}}
\newcommand{\Ha}{\frac{1}{2}}
\newcommand{\K}{{S^1}}
\newcommand{\Dc}{\bar{D}^2}

\newcommand{\A}{\mathcal{A}}
\newcommand{\F}{\operatorname{Fill}}

\newcommand{\tn}[1]{\textnormal{#1}}
\newcommand{\tb}[1]{\textbf{#1}}
\newcommand{\tr}{\textrm{d}}
\newcommand{\Ja}{\tn{\tb{J}}}
\newcommand{\e}{\mathrm{e}}
\makeatletter

\newcommand{\Rom}[1]{\expandafter\@slowromancap\romannumeral #1@}
\makeatother
\title[Quadratic isoperimetric constants of normed spaces]{Rigidity of the Pu inequality and quadratic isoperimetric constants of normed spaces}
\author{Paul Creutz}
\address{Paul Creutz, Mathematisches Institut der Universit\"at zu K\"oln, Weyertal 86-90, 50931 K\"oln, Germany}
\email{pcreutz@math.uni-koeln.de}
\thanks{The author was partially supported by the DFG grant SPP 2026.}
\begin{document}
\begin{abstract}
Our main result gives an improved bound on the filling areas of closed curves in Banach spaces which are not closed geodesics. As applications we show rigidity of Pu's classical systolic inequality and investigate the isoperimetric constants of normed spaces. The latter has further applications concerning the regularity of minimal surfaces in Finsler manifolds.
\end{abstract}
\maketitle
\section{Introduction}
\subsection{Rigidity of the Pu inequality}
Let $d:\K\times \K\rightarrow \R$ be a metric on the circle. Then the \emph{filling area} of $d$ is defined as
\begin{equation}
\label{eq1}
\tn{Fill}(d):=\inf_{g} \left\{ \tn{Area}(g)\right\}
\end{equation}
where $g$ ranges over all Riemannian metrics $g$ on the disc $D^2$ such that the boundary distance function $\tn{bd}_g:\K\times \K \rightarrow \R$ satisfies $\tn{bd}_g\geq d$. We call a Riemannian metric~$g$ a \emph{minimal filling} if its area equals the filling area of its boundary distance function. Ivanov proved in~\cite{Iva01} that \emph{simple} metrics are minimal fillings. By~\cite{PU05} simple metrics satisfy a strong rigidity property. The round metric on the hemisphere~$H^2$ is only simple in the interior. However still the Pu inequality,~\cite{Pu52}, tells us that it is a minimal filling. In particular
\begin{equation}
\tn{Fill}(d_{\K})=2\pi
\end{equation}
where $d_{\K}$ denotes the angular metric on $\K$. One of our main results is rigidity of the Pu inequality in the following sense.
\begin{Thm}
\label{thm1.1}
Let $d$ be a metric on $\K$ such that $d\leq d_\K$ and $d\neq d_\K$. Then \[\tn{Fill}(d)<2\pi.\]
\end{Thm}
Theorem~\ref{thm1.1} might seem non surprising at first glance. However it contrasts the non-rigidity of the Besicovitch inequality,~\cite{Bes52}, which has been noted in~\cite{BI16}. In general it is a seemingly hard question to describe the metrics that arise as boundary distance functions of Riemannian metrics. Some understanding in the more flexible Finsler setting has been provided in~\cite{BI16}.\par 
If one moves to higher dimensions or does not restrict to disc type surfaces in~\eqref{eq1} then simple metrics and the round metric on~$H^n$ are still conjectured to be minimal fillings. This is subject to famous conjectures of Burago-Ivanov and Gromov,~\cite{BI10,Gro83}, which are still widely open. See~\cite{BCIK05,BI10,Iva13,BCG95,BI13} for some progress.\par
More generally we study filling areas of curves in Banach spaces. To this end we fix a Banach space $X$ and an area functional in the sense of convex geometry~$\mathcal{A}$, see Section~\ref{subsec22} below for the definition. For the moment the reader may think of the parametrized Hausdorff measure~$\A^b$, also called \emph{Busemann area functional}. It is given for a Lipschitz disc $f:D^2\rightarrow X$ by
\begin{equation}
\label{eq2}
\mathcal{A}^b(f):=\int_X \tn{card}\left( f\inv (y) \right)\ \tr \mathcal{H}^2_X(y).
\end{equation}
The \emph{filling area} $\tn{Fill}^\mathcal{A}(\gamma)$ of a closed curve $\gamma$ in~$X$ is defined as the infimum of~$\mathcal{A}(f)$ where~$f$ ranges over all Lipschitz discs spanning~$\gamma$. The main result of~\cite{Cre20} states: \emph{if $\gamma:(S^1,d_{S^1})\rightarrow X$ is $1$-Lipschitz then~$\gamma$ extends to a $1$-Lipschitz map~$m:H^2\rightarrow X$.}
We refine this result as follows.
\begin{Thm}
\label{thm1.2}
If $\gamma$ is $1$-Lipschitz but not an isometric  embedding then $m$ is area decreasing. In particular $\tn{Fill}^\mathcal{A}(\gamma)<2\pi$.
\end{Thm}
It follows from~\cite[Theorem~$0.4$]{Iva09} that Theorem~\ref{thm1.1} is in fact a special case of Theorem~\ref{thm1.2}. It corresponds to choosing~$X=\ell^\infty$ and~$\mathcal{A}=\mathcal{A}^{ir}$ as Ivanov's inscribed Riemannian area functional. Some other applications of Theorem~\ref{thm1.2} are discussed in Sections~\ref{subsec12} and~\ref{subsec13} below. Beyond these Theorem~\ref{thm1.2} seems also amenable to future applications in cut and paste arguments in the context of systolic inequalities and similar problems.
\subsection{Quadratic isoperimetric spectra}
\label{subsec12}
The isoperimetric profile or (geometric) \emph{Dehn function}~$\delta_X^\mathcal{A}:(0,\infty)\rightarrow [0,\infty]$ of a metric space~$X$ is defined by \begin{equation}
\delta_X^\mathcal{A}(r):=\sup_\gamma\left\{ \tn{Fill}^\mathcal{A}(\gamma)\right\}
\end{equation}
where $\gamma$ ranges over all closed Lipschitz curves in~$X$ such that~$\ell(\gamma)\leq r$. If~$X$ is simply connected and satisfies some weak geometric assumptions then the asymptotic growth of the geometric Dehn function is the same as that of the combinatorial Dehn function of a group acting geometrically on~$X$, see~\cite{Gro83,LWYar}. The latter is a well-studied quasi-isometry invariant in geometric group theory. The isoperimetric spectrum~$\tn{IP}$ is defined as the set of those~$\alpha$ in~$[1,\infty)$ such that there is is a finitely presentable group with asymptotic growth~$\simeq r^\alpha$. By~\cite{Gro83} and~\cite{BB00} its closure is given as
\begin{equation}
\label{eq3}
\overline{\tn{IP}}=\{1\}\cup [2,\infty).
\end{equation}
The gap in~\eqref{eq3} extends into the case of quadratic growth by the following result of Wenger, \cite{Wen08}: \emph{If~$X$ is a proper geodesic metric space such that
\begin{equation}
\label{eq4}
\limsup_{r\rightarrow \infty} \frac{\delta_X^b(r)}{r^2}<\frac{1}{4\pi}
\end{equation}
then $X$ is Gromov hyperbolic}, and hence the asymptotic of its Dehn function in fact even linear. This result is sharp as the Dehn function~$\delta$ of the Euclidean space~$\R^n$ is given by~$\delta(r)=\tfrac{1}{4\pi}\cdot r^2$ independently of~$\mathcal{A}$, see Example~\ref{rem4.3}. The implications of a non-strict inequality in~\eqref{eq4} have been investigated in~\cite{Wen19}.\par
In the present paper we study the following finer non-coarse quantity:
\begin{equation}
C^\mathcal{A}(X):=\sup_{r\in (0,\infty)} \frac{\delta_X^\mathcal{A}(r)}{r^2} \in [0,\infty].
\end{equation}
We call $C^\mathcal{A}(X)$ the \emph{$\mathcal{A}$-quadratic isoperimetric constant} of~$X$. Its investigation may be motivated by the following remarkable result due to Lytchak-Wenger,~\cite{LW18}: \emph{A proper geodesic metric space~$X$ is~$\tn{CAT}(0)$ if and only if}~\begin{equation}
C^b(X)\in \left\{0,\frac{1}{4\pi}\right\}.
\end{equation}
The \emph{$\mathcal{A}$-quadratic isoperimetric spectrum} $\tn{QIS}^\mathcal{A}(\tb{M})$ of a class of metric spaces~$\tb{M}$ is the set of $\mathcal{A}$-quadratic isoperimetric constants of its elements.
Improving~\cite[Theorem 1.2]{Cre20} we are able to give a full description of the quadratic isoperimetric spectrum of the class of all Banach spaces.
\begin{Thm}
\label{thm1.3}
Let $\tn{\tb{Ban}}$ be the class of Banach spaces. Then
\begin{equation}
\label{eq5}
\tn{QIS}^b(\tn{\tb{Ban}})=\Big\{0\Big\}\cup \Big[\frac{1}{4\pi},\frac{1}{2\pi}\Big].
\end{equation}
\end{Thm}
For $\tn{\tb{Ban}}$ the quadratic isoperimetric spectrum is the same for the most reasonable choices of area functional~$\mathcal{A}$ and given by~\eqref{eq5}. However for the class $\tn{\tb{Ban}}_n$ of normed spaces of fixed finite dimension~$n$ the quadratic isoperimetric spectrum very much depends on the choice of area functional. Beyond the Busemann area functional the most popular area functionals are Benson-Gromov mass* area functional $\mathcal{A}^{m*}$ commonly used in geometric measure theory due to its strong convexity properties and the Holmes-Thompson area functional $\mathcal{A}^{ht}$ which is very natural from the point of view of Finsler geometry. For $n=2$ the spectra of the aforementioned functionals may be determined from classical results in convex geometry as:
\begin{center}
 \renewcommand{\arraystretch}{2.0}
 \begin{tabular}{|c |c |c |c| c|}
 \hline
 \ & $\mathcal{A}=\mathcal{A}^{ht}$ & $\mathcal{A}=\mathcal{A}^{b}$& $\mathcal{A}=\mathcal{A}^{m*}$ &$\mathcal{A}=\mathcal{A}^{ir}$\\ [0.5ex] 
 \hline
 $\tn{QIS}^{\mathcal{A}}(\tn{\tb{Ban}}_2)=...$ & $\left\{\frac{1}{4\pi}\right\}$ & $\left[\frac{1}{4\pi},\frac{\pi}{32}\right]$ & $\left[\frac{1}{4\pi},\frac{1}{8}\right]$&$\left[\frac{1}{4\pi},\frac{1}{8}\right]$ \\
 \hline
\end{tabular},
\end{center}
\ \\
see Example~\ref{ex4.7}. In particular every two dimensional normed space satisfies the Euclidean isoperimetric inequality with respect to~$\mathcal{A}^{ht}$ while it satisfies the Euclidean isoperimetric inequality with respect to~$\mathcal{A}^b$ only if it is Euclidean. 
The following result clarifies the behaviour between dimensions~$2$ and~$\infty$.
\begin{Thm}
\label{thm1.4}
Let $n \geq 2$ and $\mathcal{A}$ be an area functional such that $\mathcal{A}\geq\mathcal{A}^{ht}$. Then the $\mathcal{A}$-quadratic isoperimetric spectrum of~$\tn{\tb{Ban}}_n$ is a compact interval~$[\frac{1}{4\pi},r^\mathcal{A}_n]$ where~$r_n^\mathcal{A}<\frac{1}{2\pi}$ is nondecreasing in~$n$ and converges to $\frac{1}{2\pi}$ as $n\rightarrow \infty$.
\end{Thm}
The assumption $\mathcal{A}\geq \mathcal{A}^{ht}$ is satisfied for $\mathcal{A}=\mathcal{A}^b,\mathcal{A}^{ht},\mathcal{A}^{m*},\mathcal{A}^{ir}$, see Section~\ref{subsec22}. 
Explicit values of the optimal constants~$r_n^\A$ beyond the aforementioned case $n=2$ remain open. It is natural to think of our setting as the isoperimetric problem in dimension one. In the case $n=2$ we benefit of the coincidence of dimension one and \emph{co}dimension one. The latter is the mostly studied situation and essentially solved in finite dimensional normed spaces as well as many other classes of spaces, see for example~\cite{APT04,Kle92,Cro84,MR02,Rit12}... Beyond dimension one and codimension one isoperimetric inequalities have been obtained in~\cite{Gro83,AK00,Wen05}. However sharp constants are only known in Euclidean space and very few other situations, see~\cite{Alm86,Sch18}.
\subsection{Minimal surfaces in Finsler manifolds}
\label{subsec13}
Let $X$ be a proper metric space which satisfies a local quadratic isoperimetric inequality and $\Gamma$ a rectifiable Jordan curve in~$X$. 
Set $\Lambda(\Gamma,X)$ to be the set of those Sobolev discs~$u\in W^{1,2}(D^2,X)$ for which the trace $u_{|\K}$ gives a monotone parametrization of~$\Gamma$. The following solution of the Plateau problem has been given by Lytchak-Wenger in~\cite{LW17a,LW17b}:\emph{ if $\Lambda(\Gamma,X)\neq \emptyset$ then there is $u\in \Lambda(\Gamma,X)$ of least parametrized Hausdorff measure which moreover may be chosen infinitesimally isotropic.} Such $u$ will be called a \emph{solution of the Plateau problem}. Variants of the metric space valued Plateau problem have been solved for collections of Jordan curves and surfaces of higher genus in~\cite{FW19} and for self-intersecting curves in~\cite{Cre19}.\par 
For a solution of Plateau's problem~$u$ a factorization $u=\bar{u}\circ P$ with the following properties has been investigated in~\cite{LW18a}:
\begin{itemize}
\item $Z_u$ is a geodesic metric space homeomorphic to $D^2$,
\item $P:D^2\rightarrow Z_u$ is monotone,
\item  and $\bar{u}:D^2\rightarrow X$ is $1$-Lipschitz.
\end{itemize}
An analytically more well-behaved variation of this factorization has been discussed in~\cite{CS19}. In general the branch set of $u$ may be large and the map $P_u$ highly non-injective, see~\cite[Example~11.3]{LW18a}. However Question~$11.4$ in~\cite{LW18a} asks:  \emph{Can the set of branch points of a solution of the
Plateau problem be large if the isoperimetric constant~$C$ is smaller than~$\frac{1}{2\pi}$? Can the map~$P$ be non-injective in this case?} By Theorem~\ref{thm1.4} a positive answer to this question would apply in the case that~$X$ is a finite dimensional normed space or a compact Finsler manifold. This would be desirable as up to now the branch set of solutions of the Plateau problem in Finsler manifolds can only be controlled under restrictive assumptions on~$X$ and~$\Gamma$, see~\cite[Theorem~$1.6$]{OvdM15}.\par 
The quadratic isoperimetric constant of~$X$ also controls the Hölder regularity of solutions of the Plateau problem and more general $X$-valued (quasi-)harmonic discs, see~\cite{LW17a,LW16}. 
In particular Theorem~\ref{thm1.4} may be applied to improve the $\alpha$-Hölder regularity of solutions of the Plateau problem in Finsler manifolds calculated in~\cite[Theorem~$1.4$]{Cre20} beyond the threshold case to $\alpha>\frac{\pi}{8}$. Similar calculations lead to concrete uniform Hölder constants for minimal surfaces in Finsler manifolds in the settings of~\cite{LW16}, \cite{FW19} and~\cite{Cre19}.
\subsection{Outline of proof and byproducts}
In this subsection we shortly discuss the main ideas entering in the proofs of Theorems~\ref{thm1.2} and~\ref{thm1.4}. For sake of simplicity we restrict here to discussing the Holmes-Thompson area functional~$\mathcal{A}^{ht}$. All quantities in this subsection shall be understood with respect to this choice of area functional.\par 
We start with the proof of Theorem~\ref{thm1.2}. For sake of simplicity we restrict here to the case that~$X$ is finite dimenisonal. Let $p \in H^2$ be a point of differentiability of~$m$ and let $v\in T_p H^2$ and $q,\bar{q}\in \K$ be the endpoints of the grand arc passing through~$p$ in direction~$v$. If $||d_p m(v)||=|v|$ then there is $\Lambda \in X^*$ satisfying $||\Lambda||=1$ and $\Lambda(d_pf(v))=|v|$. A somewhat analytical argument involving the optimal transport going on in the proof of the main result of \cite{Cre20} then shows that \[(\Lambda\circ \gamma)(\bar{q})-(\Lambda\circ \gamma)(q)=\pi\] and hence~$\gamma$ restricted to $\{q,\bar{q}\}$ is isometric. By assumption this cannot hold for all~$q$ and in particular~$m$ must be infinitesimally shrinking in some direction. As the Holmes-Thompson area functional is sensitive to such infinitesimal changes the map~$m$ is area decreasing. The proof for general Banach spaces and area functionals that we perform below is conceptually similar but more technical.\par
To prove Theorem~\ref{thm1.4} we endow $\tn{\tb{Ban}}_n$ with the Banach-Mazur distance. Then $C(X)$ is continuous in~$X$ and hence the quadratic isoperimetric spectrum of~$\tn{\tb{Ban}}_n$ is a compact interval~$[l_n,r_n]$. It then follows from~\cite{HT79} and~\cite{BI02} that $l_n= l_2= \frac{1}{4\pi}$. In order to show
\begin{equation}
\label{eq6}
r_n<\frac{1}{2\pi}
\end{equation}
we prove the existence of extremal curves in the following sense.
\begin{Lem}
\label{lem1.5}
Let $X$ be a finite dimensional normed space. Then there is a bi-Lipschitz embedding $\gamma:\K\rightarrow X$ satisfying
\begin{equation}
\tn{Fill}(\gamma)=C(X)\cdot \ell(\gamma)^2.
\end{equation}
\end{Lem}
Note that counterintuitively these extremal curves are planar only if $C(X)=\frac{1}{4\pi}$. The concrete shape of such curves remains mysterious except for very particular cases, see Example~\ref{rem4.3}. To prove~\eqref{eq6} we choose $X\in \tn{\tb{Ban}}_n$ such that $C(X)=r_n$ and within $X$ an extremal curve~$\gamma$. Without loss of generality we may assume that $\ell(\gamma)=2\pi$ and~$\gamma$ is $1$-Lipschitz. \eqref{eq6} is then implied by Theorem~\ref{thm1.2} and the following Lemma.
\begin{Lem}
\label{lem1.6}
Let $X$ be a finite dimensional normed space. Then there is no isometric embedding of $(\K,d_{\K})$ into $X$.
\end{Lem}
The proof of Lemma~\ref{lem1.6} relies on an explicit description of geodesics in $X$ in terms of the structure of its unit ball that we give below. From this characterization it follows that if $\gamma$ is an isometric embedding then the derivative $\gamma':\K \rightarrow X$ would be a measurable function which is 'too' discontinuous.\par 
Non surprisingly it is a hard task to give lower bounds on the filling areas of curves. Our main tool at hand is a generalization of the Pu inequality due to Sergei Ivanov, which implies: \emph{if $\gamma:(\K,d_{S^1}) \rightarrow X$ is an isometric embedding then}
\begin{equation}
\label{eq7}
\tn{Fill}(\gamma)\geq 2\pi,
\end{equation}
see~\cite{Iva11,Iva09}. Lemma~\ref{lem1.6} seemingly indicates that \eqref{eq7} cannot be applied. However we can still embed isometrically large finite portions of $\K$ into $\R^n_\infty$. Such embeddings together with a homotopy argument invoking \eqref{eq7} imply
\begin{equation}
\label{eq8}
r_n\geq C(\R^n_\infty)\geq \left(1-\frac{4}{n}\right)\cdot \frac{1}{2\pi}.
\end{equation}
Note that the lower bound~\eqref{eq8} leading to the asymptotic behaviour of $r_n$ is explicit while the upper bound~\eqref{eq6} is obtained by contradiction.
\subsection{Organization}
In Section~\ref{sec2} we recall some basic facts and set up notation. First in Section~\ref{subsec21} we state a characterization of the John ellipse that will be needed in the proof of Theorem~\ref{thm1.2}. Then in Section~\ref{subsec22} we recall the notion of area functionals and discuss different examples and comparison result between them. Finally in Section~\ref{subsec23} we discuss some basic homotopy arguments and their applications. Section~\ref{sec3} is dedicated to the proof of Theorem~\ref{thm1.2}. In Section~\ref{subsec31} we perform the proof modulo a somewhat more technical proposition. This proposition is proven in Sections~\ref{subsec32} and~\ref{subsec33}. To this end we also have to recall the construction of the majorization map~$m$ and discuss optimal transport plans on~$\K$. In the remaining Section~\ref{sec4} we perform the proof of Theorem~\ref{thm1.4}. First in Sections~\ref{subsec41} and~\ref{subsec42} we prove Lemma~\ref{lem1.6} and~\ref{lem1.5} respectively. Then in Section~\ref{subsec43} we discuss the quadratic isoperimetric spectra of~$\tn{\tb{Ban}}_n$ for general area functionals. Finally we complete the proofs of Theorem~\ref{thm1.3} and~\ref{thm1.4} in Section~\ref{subsec44} by discussing lower bounds such as~\eqref{eq8}.
\subsection*{Acknowledgements}
I would like to thank my PhD advisor Alexander Lytchak for great support in everything. Also I would like to thank Stefan Wenger for teaching an enlightening mini-course on "Dehn functions \& large scale geometry" at the Young Geometric Group Theory winter school which took place in Les Diablerets in~2018. Furthermore I would like to thank Vladimir Zolotov for giving me a short but very helpful explanation of optimal transport in dimension one.
\section{Preliminaries}
\label{sec2}
\subsection{The John ellipse}
\label{subsec21}
Let $B$ be a compact, convex, centrally symmetric subset of~$\R^2$ which contains the origin in its interior. Then the John ellipsoid theorem states there is a unique ellipse~$E$ of maximal volume contained in~$B$. This ellipse is called the \emph{John ellipse} and satisfies~$E\subseteq B\subseteq \sqrt{2}\cdot E$. We will need the following characterization.
\begin{Thm}
\label{thm2.1}
Let $(X,|.|)$ be an Euclidean plane and $||.||$ a norm on~$X$ which satisfies~$||.||\leq |.|$. Denote by~$E$ the unit ball of~$|.|$ and by~$B$ the unit ball of~$||.||$. Then~$E$ is the John ellipse of~$B$ iff there exist~$v_0,v_1,v_2\in X$ such that $|v_i|=||v_i||=1$ and 
\begin{align}
\label{eq9}
&\langle v_i,v_{i+1} \rangle \geq 0&&; i=0,1,2.
\end{align}
where we define $v_3=\tau(v_0):=-v_0$ as the antipodal point of~$v_0$.
\end{Thm}
\begin{proof}
We may assume without loss of generality that $X=\R^2$ and $|.|$ is the standard Euclidean norm.
Then $D^2$ is the John ellipse of~$B$ iff there are $v_0,v_1,v_2 \in \R^2$ such that $|v_i|=||v_i||=1$ and $\lambda_0,\lambda_1,\lambda_2>0$ for which
\begin{equation}
\label{eq10}
I_2=\lambda_0\cdot v_0 \otimes v_0+\lambda_1\cdot v_1 \otimes v_1 +\lambda_2\cdot v_2 \otimes v_2,
\end{equation}
see \cite{GS05}. Here $I_2\in GL_2$ denotes the identity matrix. We show that the latter condition is equivalent to~\eqref{eq9}. In either case we may assume $v_0=(1,0)$ and furthermore that $v_0,v_1,v_2$ are cyclically ordered and contained in the upper half plane. Set $a_0:=d_\K(v_1,v_2)$, $a_1:=d_\K(v_2,\tau(v_0))$ and $a_2:=d_\K(v_0,v_1)$. Then~\eqref{eq10} becomes
\begin{equation}
I_2=\lambda_0 \cdot \left( \begin{array}{rr}
1 &0\\
0&0

\end{array}
\right)
+\lambda_1 \cdot \left( \begin{array}{rr}
\cos^2(a_2) &\frac{1}{2}\sin(2a_2)\\
\frac{1}{2}\sin(2a_2)&\sin^2(a_2) 

\end{array}
\right)
+\lambda_2\cdot \left( \begin{array}{rr}
\cos^2(a_1) &-\frac{1}{2}\sin(2a_1)\\
-\frac{1}{2}\sin(2a_1)&\sin^2(a_1) 

\end{array}
\right).
\end{equation}
Solving this system of equations gives 
\begin{equation}
\lambda_i=\frac{\sin(2a_i)}{\sin(2a_1)\sin^2(a_2)+\sin(2a_2)\sin^2(a_1)}
\end{equation}
for $i=0,1,2$. In particular all the $\lambda_i$ are positive iff \eqref{eq9} holds.
\end{proof}
\subsection{Area functionals}
\label{subsec22}
The aim of this subsection is to shortly discuss area functionals in the sense of convex geometry. We will follow the approach of~\cite{LW17b} based on Jacobians. The reader is referred to \cite{Iva09,APT04,Ber14} for other equivalent viewpoints.\par 
Let $\Sigma$ be the set of seminorms on $\R^2$ and $\Sigma_0$ be the set of norms on $\R^2$.
\begin{Def}
A \textit{Jacobian} is a map $\Ja:\Sigma \rightarrow [0,\infty)$ fulfilling the following properties:
\begin{enumerate}[label={\arabic*.}]
\item
$\Ja(|.|)=1$ for $|.|$ the standard Euclidean norm on $\R^2$. \hfill (Normalization)
\item
$\Ja(s) \geq \Ja(s')$ whenever $s\geq s'$. \hfill (Monotonicity)
\item
$\Ja(s \circ T)=|\det{T}|\ \Ja(s)$ for $T \in M_2(\R)$. \hfill (Transformation law)
\end{enumerate}
\end{Def}
\begin{Ex}
It follows readily that $\Ja(s)=0$ if and only if $s$ is degenerate.  So it suffices to define the following examples of Jacobians on a norm $||.||\in \Sigma_0$ with unit ball $B$.
\begin{enumerate}[label={\arabic*.}]
\item The \textit{Busemann Jacobian} $\Ja^b$ is defined by
\begin{equation}
\Ja^b(||.||):=\frac{\pi}{\Leb^2(B)}
\end{equation}
where $\Leb^2$ denotes the standard Lebesgue measure on $\R^2$.
\item The \textit{Holmes-Thompson Jacobian} $\Ja^{ht}$ is defined by 
\begin{equation}
\Ja^{ht}(||.||):=\frac{\Leb^2(B^*)}{\pi}
\end{equation}
where $B^*:=\{v \in \R^2|\ \langle v,w\rangle \leq 1 ;\forall w \in B\}$ is the polar body of $B$.
\item The \textit{Benson-Gromov mass* Jacobian} is defined by
\begin{equation}
\Ja^{m*}(||.||):=\sup_P \frac{4}{\Leb^2(P)}
\end{equation}
where $P$ ranges over all paralellograms containing $B$.
\item \textit{Ivanov's inscribed Riemannian Jacobian} $\Ja^{ir}$ is defined by 
\begin{equation}
\label{eq11}
\Ja^{ir}(||.||):=\frac{\pi}{\Leb^2(E)}
\end{equation}
where $E$ is the John ellipse of $B$. Dually the \emph{circumscribed Riemannian Jacobian}~$\Ja^{cr}$ is defined by taking $E$ in~\eqref{eq11} as a ellipse of least area containing~$B$.
\end{enumerate}
\end{Ex}
\begin{Rem}
The presented Jacobians satisfy the following comparison results. 
\begin{enumerate}[label={\arabic*.}]
\item From the Blaschke-Santaló inequality and the definitions we deduce that
\begin{equation}
\label{eq12}
\Ja^{cr}(||.||)\leq \Ja^{ht}(||.||)\leq \Ja^b(||.||)\leq \Ja^{ir}(||.||)
\end{equation}
where each of the inequalities is strict iff $||.||$ is not Euclidean.
\item
Dually one can see that
\begin{equation}
\label{eq13}
\Ja^{ir}(||.||)\leq \frac{4}{\pi}\cdot \Ja^b(||.||)\leq \frac{\pi}{2}\cdot \Ja^{ht}(||.||)\leq 2\cdot\Ja^{cr}(||.||)
\end{equation}
where equality is attained iff $B$ is a parallelogram. The middle part is the Mahler-Reisner inequality, see \cite{Tho96}. The other two follow from~\cite{LW17b} and duality.
\item More generally $\Ja^{ir}$ is maximal among all Jacobians and $\Ja^{cr}$ minimal. For a Jacobian~$\Ja$ we define 
\begin{equation}
q^{\Ja}:=\inf_{||.||\in \Sigma_0} \frac{\Ja(||.||)}{\Ja^{ir}(||.||)}\in \left[\frac{1}{2},1\right].
\end{equation}
By \eqref{eq13} we have $q^{ir}=1$, $q^{ht}=\frac{2}{\pi}$, $q^{b}=\frac{\pi}{4}$ and $q^{cr}=\frac{1}{2}$. From Theorem~\ref{thm2.1} one may deduce that $q^{m*}=\frac{\sqrt{3}}{2}$ is attained if $B$ is a regular hexagon.
\item For Holmes-Thompson and mass* area functional we have
\begin{equation}
\label{eq14}
\frac{2}{\pi}\cdot\Ja^{m*}(||.||)\leq \Ja^{ht}(||.||)\leq \Ja^{m*}(||.||)
\end{equation}
where equality on the left is attained iff $B$ is a paralleolgram and equality on the right is attained iff $||.||$ is Euclidean. This follows from~\cite{APT04} and the previous observations.
\end{enumerate}
\end{Rem}
In the following let $U\subset \R^2$ be open, $X$ be a metric space and $f:U\rightarrow X$ be a Lipschitz map.
At almost every $p\in U$ the \emph{metric differential} $\tn{md}_p f$ is well defined as a seminorm on~$\R^2$ via
\begin{equation}
(\tn{md}_p f)(v):=\lim_{t\rightarrow 0} \frac{d(f(p+tv),f(p))}{|t|}.
\end{equation}
For a Jacobian $\Ja$ we define the corresponding \emph{area functional}~$\mathcal{A}^\Ja$ by setting
\begin{equation}
\label{eq15}
\A^{\Ja}(f):=\int_U \Ja(\tn{md}_p f)\ \tr \Leb^2(p).
\end{equation}
For $\mathcal{A}^b$ equation \eqref{eq15} is consistent with equation~\eqref{eq2} by a variant of the area formula, see~\cite{Kir94}.
The definition of Jacobians is cooked up as to obtain the following natural list of properties for the arising area functionals.
\begin{Lem}
\begin{enumerate}[label={\arabic*.}]
\item If $X$ is a Riemannian manifold, then $\mathcal{A}(f)=\mathcal{A}^b(f)$. \par\hfill (Normalization)
\item If $g:X\rightarrow Y$ is $L$-Lipschitz, then $\mathcal{A}(g\circ f)\leq L^2\cdot\mathcal{A}(f)$. \par \hfill (Monotonicity)
\item If $V\subseteq \R^2$ is a open and $\varphi:V\rightarrow U$ is bi-Lipschitz, then $\mathcal{A}(f\circ \varphi)=\mathcal{A}(f)$.\par \hfill (Coordinate invariance)
\end{enumerate}
\end{Lem}
In particular area functionals naturally extend to assign areas to Lipschitz maps $f:M\rightarrow X$ where $M$ is a smooth $2$-dimensional manifold.
\subsection{Homotopy arguments}
\label{subsec23}
In this section we fix an area functional $\A$ and a metric space $X$.\par 
For a Lipschitz curve $\gamma:\K\rightarrow X$ we define its \textit{$\A$-filling area} by
\begin{equation}
\F^\A(\gamma):=\inf \{\A(f)\ |\ f:\Dc\rightarrow X \tn{ Lipschitz}, f_{|\K}=\gamma\}.
\end{equation}
The crucial observation for homotopy arguments is: \emph{if $h:\K\times [0,1]\rightarrow X$ is a Lipschitz map then
\begin{equation}
|\F^\A(\gamma_0)-\F^\A(\gamma_1)|\leq \A(h)
\end{equation}
where $\gamma_i=h(\cdot,i)$.} The following lemma allows to restrict to curves which are parametrized by constant-speed in most situations.
\begin{Lem}
\label{lem2.6}
Let $\gamma_0,\gamma_1:\K \rightarrow X$ be Lipschitz curves and reparametrizations of each other. Then
\begin{equation}
\F^{\A}(\gamma_1)=\F^{\A}(\gamma_2).
\end{equation}
\end{Lem}
\begin{proof}
By Lemma~$3.6$ in \cite{LWYar} there exists a Lipschitz homotopy $h$ between $\gamma_0$ and $\gamma_1$ such that $\A(h)=0$.
\end{proof}
The following simple but useful lemma will be applied various times.
\begin{Lem}
\label{lem2.7}
Let $X$ be a geodesic metric space, $\gamma_0,\gamma_1:\K\rightarrow X$ closed Lipschitz curves and $\phi_0,...,\phi_m\in \K$ cyclically ordered points. Then
\begin{equation}
|\F^\mathcal{A}(\gamma_0)-\F^\mathcal{A}(\gamma_1)|\leq C\cdot \sum_{k=0}^m (l^0_k+l^1_k+d_k+d_{k+1})^2
\end{equation}
where $C:=C^\mathcal{A}(X)$, $d_k:=d(\gamma_0(\phi_k),\gamma_1(\phi_k))$ and $l^i_k:=\ell({\gamma_{i}}_{|[\phi_k,\phi_{k+1}]})$.
\end{Lem}
\begin{proof}
We define the Lipschitz homotopy $h:\K\times [0,1]\rightarrow X$ between $\gamma_0$ and $\gamma_1$ by setting~$h(\phi_k,\cdot)$ to be a geodesic connecting~$\gamma_0(\phi_k)$ to $\gamma_1(\phi_k)$ and filling the remaining squares by application of the quadratic isoperimetric inequality.
\end{proof}
For $L\geq 0$ let $\Gamma^L(X)$ be the set of closed $L$-Lipschitz curves in $X$ endowed with the maximum metric
\[
d_\infty(\gamma_0,\gamma_1):= \max_{\phi \in \K}\ d(\gamma_0(\phi),\gamma_1(\phi)).
\]
\begin{Cor}
\label{lem2.8}
If $X$ is geodesic and satisfies a quadratic isoperimetric inequality then $\F^\mathcal{A}:\Gamma^L(X)\rightarrow \R$ is continuous.
\end{Cor}
\begin{proof}
Let $\varepsilon:=d_\infty(\gamma_0,\gamma_1)\leq \pi\cdot L$. Choosing $m:=\lceil \frac{\pi L}{\varepsilon}\rceil$ equidistant points on~$\K$ and applying Lemma~\ref{lem2.7} gives
\begin{equation}
|\F^\mathcal{A}(\gamma_0)-\F^\mathcal{A}(\gamma_1)|\leq C\cdot m\cdot \left(\frac{4\pi L}{m}+2\varepsilon\right)^2\leq 72\pi  C  L\cdot \varepsilon.
\end{equation}
Compare also the proof of~\cite[Lemma~$18$]{Sta18}.
\end{proof}
\section{Proof of Theorem~\ref{thm1.2}}
\label{sec3}
\subsection{Reduction of the proof}
\label{subsec31}
The following majorization theorem has been obtained in~\cite{Cre20}.
\begin{Thm}
\label{thm3.1}
Let $X$ be a Banach space and $\gamma:\K\rightarrow X$ be $1$-Lipschitz. Then $\gamma$ extends to a $1$-Lipschitz map $m:H^2\rightarrow X$.
\end{Thm}
Our aim is to proof Theorem~\ref{thm1.2}. In particular that~$\mathcal{A}(m)<2\pi$ if $\gamma$ is not an isometric embedding.
\begin{Pro}
\label{pro3.2}
Let $\eta:[0,\pi]\rightarrow H^2$ be a grand-arc which is not contained in~$\K$ and~$r \neq  0, \frac{\pi}{2},\pi$. Set $p:=\eta(r)$, $v:=\eta'(r)$ and $\phi:=\eta(0)$. If $p$ is a point of metric differentiability of~$m$ and $(\tn{md}_p m)(v)=1$ then \[||\gamma(\phi)-\gamma(\tau(\phi))||=\pi\]
where~$\tau:\K \rightarrow \K$ is the antipodal map.
\end{Pro}
The proof of Proposition~\ref{pro3.2} which needs some understanding of the optimal transport going on in the construction of~$m$ will be postponed to the next subsection.
\begin{proof}[Proof of Theorem~\ref{thm1.2}]
By its maximality it suffices to prove Theorem~\ref{thm1.2} for the inscribed Riemannian area functional~$\mathcal{A}^{ir}$. For~$p\in H^2$ let $|.|_p$ be the standard norm on~$T_p H^2$ and $E_p$ its unit ball. If $p$ is a point of metric differentiability of~$m$ then we denote by $B_p\subset T_p H^2$ the unit ball of the seminorm~$\tn{md}_p m$. As $\tn{md}_p m\leq |.|_p$ it suffices to prove that the set of points $p\in H^2$ for which~$E_p$ is not the John ellipse of~$B_p$ has positive measure.\par 
As $\gamma$ is not an isometric embedding there exist~$\phi_0 \in \K$ and an open interval $I\subset \K$ containing~$\phi_0$ such that 
\begin{equation}
||\gamma(\phi)-\gamma(\tau(\phi))||<\pi
\end{equation}
for all~$\phi \in I$. For $p\in H^2\setminus \K$ and $\phi \in \K$ let $\eta_\phi:[0,\pi]\rightarrow H^2$ be the grand-arc passing through~$p$ which has $\eta(0)=\phi$ and let $v(\phi) \in T_p H^2$ be the direction of $\eta_\phi$ in $p$. Then $v$ defines a diffeomorphsm between $\K$ and the unit vectors in the tangent space at~$p$. Set $u_p:= v(\phi_0+\frac{\pi}{2})$ and $w_p:=v(\psi)$ where~$\psi$ is such that $p=\eta_\psi(\frac{\pi}{2})$. If $p\in H^2\setminus \K$ is sufficiently close to~$\phi_0$, then $w_p$ and $\tau(w_p)$ as well $v(A)$ where $A:=\K \setminus(I\cup \tau(I))$ are contained in small neighbourhoods of $u_p$ and $\tau(u_p)$. Hence Theorem~\ref{thm1.2} is implied by Theorem~\ref{thm2.1} and Proposition~\ref{pro3.2}.
\end{proof}
\subsection{The majorization map}
\label{subsec32}
 Let $(X,d)$ be a complete metric space. Denote by $\mathcal{P}(X)$ the set of separably supported Borel probability measures on $X$ and by $\mathcal{P}_1(X)$ the set of those $\mu \in \mathcal{P}(X)$ satisfying 
 \begin{equation}
\int_X d(x,y)\ \tr \mu(y)<\infty
\end{equation}  
for some $x\in X$. For a continuous map $f:X\rightarrow Y$ denote by $f_*:\mathcal{P}(X)\rightarrow \mathcal{P}(Y)$ the push forward map given by $f_*\mu(A)=\mu(f^{-1}(A))$. For $\mu,\nu\in \mathcal{P}(X)$ we call $K\in \mathcal{P}(X\times X)$ a \emph{coupling} from $\mu$ to $\nu$ if $\pi_{1*}K=\mu$ and $\pi_{2*}K=\nu$. Denote the set of couplings from $\mu$ to $\nu$ by $\Pi(\mu,\nu)$. The metric space obtained by endowing $\mathcal{P}_1(X)$ with the distance
\begin{equation}
d_W(\mu,\nu):=\inf_{K\in \Pi(\mu,\nu)} \int_{X\times X} d(x,y) \tr K(x,y)
\end{equation}
will be called the \emph{Wasserstein-$1$-space over $X$} and will also shortly be denoted by~$\mathcal{P}_1(X)$. A measurable map $T:X\rightarrow X$ will be called an \emph{optimal transport plan} from $\mu$ to $\nu$ if $T_*\mu=\nu$ and
\begin{equation}
d_W(\mu,\nu)= \int_X d(x,T(x))\ \tr \mu(x).
\end{equation}
\ \par The proof of Theorem~\ref{thm3.1} relies on four things. First for every complete metric space there is a canonical isometric embedding $\delta$ of $X$ into $\mathcal{P}_1(X)$ given by mapping $x\in X$ to the Dirac measure $\delta_x$. Secondly if $X$ is a Banach space, then there is a $1$-Lipschitz retraction $b:\mathcal{P}_1(X)\rightarrow X$ given by
\begin{equation}
b(\mu):=\int_X x \ \tr \mu(x). 
\end{equation}
The third is that if $f:X\rightarrow Y$ is $1$-Lipschitz then also the push forward map $f_*:\mathcal{P}_1(X)\rightarrow \mathcal{P}_1(Y)$ is $1$-Lipschitz. The last but most important observation is the following proposition.
\begin{Pro}
\label{pro3.3}
There is an isometric embedding $\mu:H^2\rightarrow \mathcal{P}_1(\K)$ extending the Dirac embedding $\delta:\K\rightarrow \mathcal{P}_1(\K)$.
\end{Pro}
Theorem~\ref{thm3.1} follows from these observations by setting $m:H^2\rightarrow X$ to be given by
\begin{equation}
m:=b\circ \gamma_* \circ \mu.
\end{equation}
\begin{Lem}
\label{lem3.4}
Let $\eta ,r,\phi$ be as in Proposition~\ref{pro3.2} and assume furthermore~$r<\frac{\pi}{2}$. Then there is a $C^1$-map $T=T(s,\psi):\left(0,\frac{\pi}{2}\right)\times \K \rightarrow \K$ such that for every fixed~$s$:
\begin{enumerate}
\item The map $T(s,\cdot)$ is an optimal transport plan from $\mu_{\eta(r)}$ to $\mu_{\eta(s)}$.
\item $\frac{\partial T}{\partial s}(s,\cdot)$ is positive almost everywhere on the interval~$[\phi,\tau(\phi)]$ and negative almost everywhere on~$[\tau(\phi),\phi]$.
\end{enumerate}
\end{Lem}
The proof of Lemma~\ref{lem3.4} which we postpone further to the next subsection needs some more concrete understanding of the map~$\mu$ and optimal transport plans on~$\K$.
\begin{proof}[Proof of Proposition~\ref{pro3.2}]
Set $\mu_s:=\mu_{\eta(s)}$ and let $T$ be as in Lemma~\ref{lem3.4}. Then 
\begin{equation}
\label{eq16}
(m\circ \eta)(s)=\int_\K \gamma(\psi) \ \tr \mu_s(\psi)=\int_\K \gamma(T(s,\psi)) \ \tr \mu_r(\psi).
\end{equation}
By assumption there is a sequence $s_n\searrow r$ such that 
\begin{equation}
\label{eq17}
1-\frac{1}{n}\leq \frac{||(m\circ \eta)(s_n)-(m\circ \eta)(r)||}{s_n-r}.
\end{equation}
Choose $\Lambda_n\in X^*$ such that $||\Lambda_n||=1$ and 
\begin{equation}
\label{eq18}
\Lambda_n((m\circ \eta)(s_n)-(m\circ \eta)(r))=||(m\circ \eta)(s_n)-(m\circ \eta)(r)||.
\end{equation}
Then by \eqref{eq16}, \eqref{eq17}, \eqref{eq18} and the fundamental theorem of calculus for Lipschitz functions we obtain
\begin{align}
1-\frac{1}{n}&\leq \frac{1}{s_n-r}\cdot \Lambda_n \left(\int_\K \gamma(T(s_n,\psi))-\gamma(\psi)\ \tr \mu_r(\psi)\right)\\
&=\frac{1}{s_n-r}\cdot\int_\K \int^{s_n}_r (\Lambda_n\circ \gamma)'(T(s,\psi))\cdot \frac{\partial T}{\partial s}(s,\psi)\ \tr s\ \tr\mu_r(\psi).
\end{align}
In particular there is $t_n\in [r,s_n]$ such that
\begin{equation}
\label{eq19}
1-\frac{1}{n}\leq \int_\K \underbrace{(\Lambda_n\circ \gamma)'(T(t_n,\psi))}_{=:f_n(\psi)}\cdot \underbrace{\frac{\partial T}{\partial s}(t_n,\psi)}_{=:g_n(\psi)}\ \tr\mu_r(\psi).
\end{equation}
Then as $|f_n|\leq 1$ and $g_n$ converges uniformly
\begin{equation}
\label{eq20}
\int_\K f_n\cdot g_n\ \tr \mu_r \leq \int_\K |g_n|\ \tr \mu_r \overset{n\rightarrow \infty}{\longrightarrow} \int_\K \underbrace{\left|\frac{\partial T}{\partial s}(r,\psi)\right|}_{=:g(\psi)}\ \tr \mu_r(\psi).
\end{equation}
As $\mu$ is an isometric embedding and $T(s,\cdot)$ an optimal transport plan 
\begin{equation}
\label{eq21}
\int_\K g(\psi)\ \tr \mu_r(\psi)= \int_\K \lim_{s\searrow r} \frac{|T(s,\psi)-\psi|}{s-r}\ \tr \mu_r(\psi)=\lim_{s\searrow r} \frac{d_W(\mu_s,\mu_r)}{s-r}=1.
\end{equation}
By \eqref{eq19}, \eqref{eq20} and \eqref{eq21} $f_n\cdot g_n\rightarrow g$ in $\mu_r$-measure. So up to passing to a subsequence $f_n\cdot g_n\rightarrow g$ holds $\mu_r$-almost everywhere and hence also $\mathcal{H}^1$-almost everywhere. Hence by Lemma~\ref{lem3.4} $f_n\rightarrow 1$ almost everywhere on $[\phi,\tau(\phi)]$. As the diffeomorphisms $T(t_n,\cdot)$ and their inverses are of uniformly bounded $C^1$ norm, $(\Lambda_n \circ \gamma)'\rightarrow 1$ in $\mathcal{H}^1$-measure on $[\phi,\tau(\phi)]$ and hence up to again passing to a  further subsequence the convergence holds almost everywhere on $[\phi,\tau(\phi)]$. So by the fundamental theorem of calculus and the dominated convergence theorem
\begin{equation}
||\gamma(\phi)-\gamma(\tau(\phi))||\geq \Lambda_n(\gamma(\tau(\phi))-\gamma(\phi))=\int^{\tau(\phi)}_\phi (\Lambda_n \circ \gamma)'(\psi)\ \tr \psi \overset{n\rightarrow \infty}{\longrightarrow} \pi
\end{equation}
which completes the proof.
\end{proof}
\subsection{Optimal transport plans}
\label{subsec33}
For intervals we have the following simple description of optimal transport plans, cf.~\cite{Vil03}.
\begin{Lem}
\label{lem3.5}
Let~$I\subset\R$ be a closed interval and $\mu,\nu\in \mathcal{P}_1(I)$ be absolutely continuous measures with strictly increasing distribution functions $F_\mu, F_\nu:I\rightarrow [0,1]$. Then an optimal transport map from $\mu$ to $\nu$ is given by $T=F_\nu^{-1}\circ F_\mu$ and
\begin{equation}
\label{eq22}
d_W(\mu,\nu)=\int_I |F_\mu(s)-F_\nu(s)|\tr s=\int_0^1 |F_\mu^{-1}(t)-F_\nu^{-1}(t)|\tr t.
\end{equation}
\end{Lem}
We are however mainly interested in understanding the Wasserstein distance on~$\K$. This problem has been solved by Cabrelli-Molter in~\cite{CM95} and we shortly discuss their approach here.\par 
For a point $\phi\in \K$ and absolutely continuous $\mu \in \mathcal{P}^1(\K)$ we denote by $\mu^\phi$ the measure on $[0,2\pi]$ which corresponds to $\mu$ under the orientation preserving 'identification' of $[0,2\pi]$ and $\K$ mapping $0$ to $\phi$. Then for $\mu,\nu\in \mathcal{P}_1(\K)$ the inequality
\begin{equation}
d_W(\mu,\nu)\leq d_W(\mu^\phi,\nu^\phi)
\end{equation}
is immediate. For $\mu,\nu\in \mathcal{P}_1(\K)$ we call $\phi\in \K$ an \emph{equilibrated cutpoint} for $(\mu,\nu)$ if there is a Borel partition $[0,2\pi)=A\dot\cup B$ such that $|A|=|B|$, $F_{\mu^\phi} \leq F_{\nu^\phi}$ on $A$ and $F_{\mu^\phi} \geq F_{\nu^\phi}$ on~$B$. This definition is justified by the following theorem.
\begin{Thm}[\cite{CM95}]
\label{thm3.6}
Let $\mu,\nu\in \mathcal{P}_1(\K)$ be absolutely continuous measures. Then there exists an equilibrated cutpoint $\phi$ for $(\mu,\nu)$ and for every such $\phi$ one has 
\begin{equation}
d_W(\mu,\nu)=d_W(\mu^\phi,\nu^\phi).
\end{equation}
\end{Thm}
So calculating the Wasserstein distance between distributions on $\K$ amounts to finding an equilibrated cutpoint and then calculating the integral~\eqref{eq22}.\par 
The construction of the map~$\mu$ in Proposition~\ref{pro3.3} goes as follows. For fixed $p\in H^2\setminus \K$ let \[d_p:\K\rightarrow \R, \psi \mapsto d_{S^2}(p,\psi)
\]
be the distance function to $p$. Let $b_p \in \K$ be such that $d_{S^2}(p,b_p)=d_{S^2}(p,\K)$ and $k_p:=\cos(d_{S^2}(p,b_p))$. Let $h_p:\K \rightarrow \R$ be given by
 \begin{equation}
 h_p(\psi):=\Ha\left(d_p''(\psi)\right)^++\frac{1-k_p}{2\pi}
\end{equation}
and let $\mu_p$ be the measure on $\K$ which is absolutely continuous with density $h_p$. The most technical part in the proof of Proposition~\ref{pro3.3} amounts to the following lemma, compare Sections~$3.2$ and~$3.3$ in~\cite{Cre20}.
\begin{Lem}
\label{lem3.7}
Let $\eta,\phi$ be as in Proposition~\ref{pro3.2}.
\begin{itemize}
\item If $r,s\in (0,\pi)$ then $\phi$ is a balanced cutpoint for $(\mu_{\eta(s)},\mu_{\eta(r)})$.
\item If furthermore $r\leq s\leq \frac{\pi}{2}$ then a corresponding Borel partition is given by $A=[0,\pi)$ and $B=[\pi,2\pi]$.
\end{itemize}
\end{Lem}
\begin{proof}[Proof of Lemma~\ref{lem3.4}]
Set $\mu_s:=\mu_{\eta(s)}$, $d_s:=d_{\eta(s)}$, $k_s:=k_{\eta(s)}$, $h_s:=h_{\eta(s)}$, $b_s:=b_{\eta(s)}$.\par 
Identify $\K$ and $[0,2\pi]$ such that~$\phi$ corresponds to~$0$. We set $\mathcal{D}:=\left(0,\frac{\pi}{2}\right)\times [0,2\pi]$ and define the analytic function $d:\mathcal{D}\rightarrow \R$ by $d(s,\psi):=d_{s}(\psi).$ 
Furthermore we define $F:\mathcal{D} \rightarrow [0,1]$ by
\begin{equation}
F(s,\psi)=F_s(\psi):=\mu_{s}([0,\psi))=\int^\psi_0 h_{s}(\varphi)\tr \varphi.
\end{equation}
Then by Lemma~\ref{lem3.5}, Theorem~\ref{thm3.6} and Lemma~\ref{lem3.7} an optimal transport map $T^{s}$ from $\mu_r$ to $\mu_s$ is given by
\begin{equation}
T^s(\psi)=F_s^{-1}(F_r(\psi)).
\end{equation}
Furthermore if $s\geq t$, then $T^s(\psi)\geq T^t(\psi)$ for $\psi\in [0,\pi]$ and $T^s(\psi)\leq T^t(\psi)$ for $\psi\in [\pi,2\pi]$. We define the map~$T:\mathcal{D}\rightarrow [0,2\pi]$ by~$T(s,\psi):=T^s(\psi)$.\par
Let $\xi$ be the angle between~$\eta$ and~$\K$ in $\phi$. Then by the spherical sine theorem we have that
\begin{equation}
k(s):=k_s=\sqrt{1-\sin^2(s)\sin^2(\xi)}
\end{equation}
is an analytic function with nonzero derivative on~$\left(0,\frac{\pi}{2}\right)$. We have~$b_s\in (-\frac{\pi}{2},\frac{\pi}{2})$ and hence by~\cite[Lemma~$3.1$]{Cre20} and the first variation formula
\begin{equation}
F(s,\psi)=\Ha \cdot \begin{cases}
\frac{\partial d}{\partial \psi} (s,\psi)-\cos(\xi)+\psi \cdot \frac{1-k_s}{\pi}&; \ 0\leq \psi\leq b_s+\frac{\pi}{2}\\
\frac{\partial d}{\partial \psi} \left(s,b_s+\frac{\pi}{2}\right)-\cos(\xi)+\psi \cdot \frac{1-k_s}{\pi}&; \ b_s+\frac{\pi}{2}\leq \psi\leq b_s+\frac{3\pi}{2}\\
\frac{\partial d}{\partial \psi} (s,\psi)+2\cdot\frac{\partial d}{\partial \psi} \left(s,b_s+\frac{\pi}{2}\right)-\cos(\xi)+\psi \cdot \frac{1-k_s}{\pi}&; \ b_s+\frac{3\pi}{2}\leq \psi\leq 2\pi.
\end{cases}
\end{equation}
By \cite[Lemma~$3.1$]{Cre20} 
\begin{equation}
\frac{\partial^2 d}{(\partial \psi)^2} \left(s,b_s+\frac{\pi}{2}\right)=0
\end{equation}
implying that $\frac{\partial F}{\partial s}$ is well defined and continuous on~$\mathcal{D}$. Also
\begin{equation}
\frac{\partial F}{\partial \psi}(s,\psi)=h_s(\psi)
\end{equation}
is continuous on~$\mathcal{D}$ and hence $F\in C^1(\mathcal{D})$. By the first variation formula
\begin{equation}
\frac{\partial d}{\partial \psi} \left(s,b_s+\frac{\pi}{2}\right)=k_s.
\end{equation}
So for fixed parameter~$s$ the function~$\frac{\partial F}{\partial s}(s,\cdot):[0,2\pi]\rightarrow \R$ is piecewise analytic on the intervals~$\left[0,b_s+\frac{\pi}{2}\right],\left[b_s+\frac{\pi}{2},b_s+\frac{3\pi}{2}\right],\left[b_s+\frac{3\pi}{2},2\pi\right]$ and on~$\left[b_s+\frac{\pi}{2},b_s+\frac{3\pi}{2}\right]$ it is given by
\begin{equation}
\frac{\partial F}{\partial s}(s,\psi)=\Ha \cdot \frac{\partial k}{\partial s}(s)\cdot \left(1-\frac{\psi}{\pi}\right)
\end{equation}
which is zero only for~$s=\pi$. In particular $\frac{\partial F}{\partial s}(s,\cdot)$ has only finitely many zeros.\par
For fixed $s$ the map $F_s:[0,2\pi]\rightarrow [0,1]$ defines a $C^1$-diffeomorphism and hence by the implicit function theorem
\begin{equation}
\frac{\partial }{\partial s} \left(F_s\inv(v)\right)=\frac{-1}{\frac{\partial F}{\partial \psi}(s,F_s\inv(v))}\cdot \frac{\partial F}{\partial s} (s,F_s\inv(v)).
\end{equation}
Hence $T$ is differentiable in $s$ and
\begin{equation}
\frac{\partial T}{\partial s}(s,\psi)=-\frac{\frac{\partial F}{\partial s}(s,T^{s}(\psi))}{h_{s}(T^{s}(\psi))}
\end{equation}
is continuous on $\mathcal{D}$. In particular $\frac{\partial T}{\partial s}(s,\psi)=0$ iff $\frac{\partial F}{\partial s}(s,T^{s}(\psi))=0$. Implying that $\frac{\partial T}{\partial s}(s,\cdot)$ has only finitely many zeros. To complete the proof that $T$ is $C^1$ it suffices to note that
\begin{equation}
\frac{\partial T}{\partial \psi}(s,\psi)=\frac{-h_{r}(\psi)}{h_{s}(T^s(\psi))}
\end{equation}
is continuous on $\mathcal{D}$.
\end{proof}
\section{Quadratic isoperimetric spectra}
\label{sec4}
\subsection{Geodesics in finite dimensional normed spaces}
\label{subsec41}
In this subsection fix a finite dimensional normed space~$(X,||.||)$. The following lemma characterizes geodesics in $X$ in terms of the shape of the unit ball.
\begin{Lem}
\label{lem4.1}
Let $\gamma:[a,b]\rightarrow X$ a $1$-Lipschitz curve connecting the points $p$ and $q$. Set $v:=q-p$ and let $\Lambda \in X^*$ be such that $||\Lambda||=1$ and $\Lambda(v)=||v||$. 
Then $\gamma$ is an isometric embedding iff $\Lambda(\gamma'(t))=1$ for almost every $t \in [a,b]$.
\end{Lem}
\begin{proof}
Applying the fundamental theorem of calculus to the Lipschitz function $\Lambda \circ \gamma:[a,b]\rightarrow \R$ we get
\begin{equation}
||v||=\Lambda\left(v\right)
=(\Lambda\circ\gamma)(b)-(\Lambda\circ \gamma)(a)
=\int^b_{a}\Lambda( \gamma'(t)) \ \tr t 
\label{e.4.1}
\leq b-a.
\end{equation}
So $\gamma$ is an isometric embedding iff $||v||=b-a$ iff $\Lambda( \gamma'(t))=1$ almost everywhere in~$[a,b]$.
\end{proof}
For the proof of Lemma~\ref{lem1.6} we remind the reader that a metric space valued function $f:\R^m \rightarrow Y$ is called \emph{approximately continuous} at $x \in \R^m$ if for every $\epsilon >0$ 
\[
\lim_{r\downarrow 0}\frac{\Leb^m\left(\left\{y \in B_r(x): d(f(y),f(x))\geq \epsilon \right\}\right)}{\Leb^m\left(B_r(x)\right)}=0.
\]
If $Y$ is a separable metric space then a Borel measurable function $f:\R^m \rightarrow Y$ is approximately continuous almost everywhere, see \cite{LT04}.
\begin{proof}[Proof of Lemma~\ref{lem1.6}]
Let $\phi,\psi \in \K$ be antipodal points such that $||\gamma(\phi)-\gamma(\psi)||=\pi$. Let $v$ and $\Lambda$ be chosen as in Lemma~\ref{lem4.1} for $p:=\gamma(\phi)$ and $q:=\gamma(\psi)$. Then it follows that $\gamma$ is the composition of a shortest path $\gamma_1$ connecting~$p$ to~$q$ and a shortest path $\gamma_2$ connecting~$q$ to~$p$. By Lemma~\ref{lem4.1} we have $\Lambda(\gamma_1'(t))=1$ almost everywhere and $\Lambda(\gamma_2'(t))=-1$ almost everywhere. In particular the measurable function $\gamma':\K \rightarrow X$ cannot be approximately continuous neither at $\phi$ nor at $\psi$. As $X$ is separable this implies the claim.
\end{proof}
\begin{Rem}
\label{rem4.2}
Clearly Lemma~\ref{lem1.6} fails for general Banach spaces. Beyond the Kuratowski embedding of~$\K$ into~$\ell^\infty$ it is also easy to write down an isometric embedding of~$\K$ into $L^1$. Note however that the proof of Lemma~\ref{lem1.6} goes through as soon as $X$ is a Banach space which has the Radon-Nikodym property such as~$\ell^1$, cf.~\cite[Chapter 5]{BL00}.
\end{Rem}
\subsection{Extremal curves}
\label{subsec42}
In this subsection we fix an area functional $\A$ and a finite dimensional normed space~$(X,||.||)$ of dimension at least two.\par
Note that Lemma~\ref{lem1.5} is stated for~$\mathcal{A}^{ht}$. We will however prove it for any area functional. To this end recall that a closed Jordan curve $\gamma$ in~$X$ is said to satisfy a \emph{chord-arc condition} with constant~$\lambda \geq 1$ if for every distinct~$v,w\in \tn{im}(\gamma)$ the shorter of the two arcs of~$\gamma$ between~$v$ and~$w$ has length bounded above by~$\lambda \cdot ||v-w||$. A Jordan curve is bi-Lipschitz to~$\K$ iff it satisfies a chord-arc condition with \emph{some} constant~$\lambda\geq 1$.
\begin{proof}[Proof of Lemma~\ref{lem1.5}]
First we note that $C:=C^\mathcal{A}(X)\in (0,\infty)$, see the proof of Theorem~\ref{thm4.6} below. For a closed and nonconstant Lipschitz curve $\gamma$ in $X$ we define 
\begin{equation}
C(\gamma):=\frac{\tn{Fill}^\mathcal{A}(\gamma)}{\ell(\gamma)^2}.
\end{equation}
By definition there exists a sequence of closed nonconstant Lipschitz curves~$\gamma_n$ such that~$C(\gamma_n)\nearrow C$. By scaling and translation invariance of $C(\gamma)$ we may assume that $\ell(\gamma_n)=1$ and that the image of the $\gamma_n$'s is contained inthe unit ball~$B$. Furthermore by Lemma~\ref{lem2.6} we may assume that the curves~$\gamma_n$ are parametrized by constant speed and hence $1$-Lipschitz. By the Arzela-Ascoli theorem after passing to a subsequence the curves $\gamma_n$ converge uniformly to a closed Lipschitz curve~$\gamma$. Corollary~\ref{lem2.8} implies
\begin{equation}
C(\gamma_n)=\tn{Fill}^\mathcal{A}(\gamma_n)\rightarrow C=\tn{Fill}^\mathcal{A}(\gamma).
\end{equation}
By lower semicontinuouity of length $\ell(\gamma)\leq 1$ and hence $C(\gamma)=C$.\par
Let $\gamma$ be such that $C(\gamma)=C$ and assume $\gamma$ is not Jordan or does not satisfy a chord-arc condition with constant~$\lambda=\sqrt{2}+1$. Then there exists $\phi_1,\phi_2\in \K$ such that the following holds: \emph{If  $l_1\leq l_2$ are the lengths of the two arcs of~$\gamma$ connecting~$\gamma(\phi_1)$ to~$\gamma(\phi_2)$ then $\lambda\cdot d<l_1$ where $d:=||\gamma(\phi_1)-\gamma(\phi_2)||$.}\par 
Then by $l_1\leq l_2$ and the particular choice of $\lambda$ we have 
\begin{equation}
(l_1+d)^2+(l_2+d)^2<l_1^2+l_2^2+\left(\frac{4}{\lambda}+ \frac{2}{\lambda^2}\right)l_1 l_2\leq (l_1+l_2)^2=\ell(\gamma)^2.
\end{equation}
Applying Lemma~\ref{lem2.7} where we take $\gamma_0:=\gamma$ and $\gamma_1$ as the curve identically constant~$\gamma(\phi_1)$ gives
\begin{equation}
\F^\mathcal{A}(\gamma)\leq C\cdot \left((l_1+d)^2+(l_2+d)^2\right)<C\cdot \ell(\gamma)^2.
\end{equation}
This contradicts the extremality of~$\gamma$.
\end{proof}
We call a unit-speed curve~$\gamma:\K\rightarrow X$ an~\emph{$\mathcal{A}$-extremal curve} if $C(\gamma)=C$.
\begin{Ex} \label{rem4.3}
The particular shape of such extremal curves $\gamma$ is only known in the following two situations.
\begin{enumerate}
\item If $X=\R^n$ is Euclidean then up to Euclidean motions the extremal curves are given by the standard embedding of~$\K$ into~$\R^2$. In particular all such curves are planar and~$C(X)=\frac{1}{4\pi}$ independently of~$\mathcal{A}$. This follows from Reshetnyak's majorization theorem, \cite{Res68}, and the existence of a $1$-Lipschitz retraction of $X$ onto any of its linear subspaces.
\item If $X$ is a $2$-dimensional normed space then there is also less ambiguity in the choice of area functional~$\mathcal{A}$. This is because all metric differentials $\tn{md}_p f$ of a Lipschitz map $f:D^2\rightarrow X$ are either degenerate or give rise to normed spaces isometric to~$X$. In particular the shape of extremal curves does not depend on~$\mathcal{A}$ and for area functionals $\mathcal{A}$ and $\bar{\mathcal{A}}$ one has
\begin{equation}
\label{eq23}
C^{\bar{\mathcal{A}}}(X)=\frac{\Ja^{\bar{\mathcal{A}}}(X)}{\Ja^\mathcal{A}(X)}\cdot C^\mathcal{A}(X).
\end{equation}
Maybe somewhat surprisingly the extremal curves~$\gamma$ do not correspond to the boundary contour of the unit ball~$B$ but rather to the boundary contour of the dual unit ball~$B^*$ under a suitable identification of $X$ and $X^*$, see~\cite{Tho96}.
\end{enumerate}
\end{Ex}
Although in these two examples the choice of $\mathcal{A}$ is immaterial for the shape of~$\gamma$ it is very likely that this phenomenon is far from being true for a generic finite dimensional normed space~$X$.
\begin{Rem}
Lemma~\ref{lem1.5} does not hold for general Banach spaces~$X$. Namely by Theorem~\ref{thm1.2} if $C^{\mathcal{A}}(X)=\frac{1}{2\pi}$ and $\gamma$ is an $\mathcal{A}$-extremal curve in~$X$ then~$\gamma$ must be an isometric embedding of~$\K$. However $C^{ht}(\ell^1)=\frac{1}{2\pi}$ by Remark~\ref{rem4.10} below and~$\ell^1$ does not admit such an isometric embedding by Remark~\ref{rem4.2}.
\end{Rem}
\subsection{Quadratic Isoperimetric Spectra}
\label{subsec43}
To prove Theorem~\ref{thm1.4} we fix~$n\geq 2$ and endow $\tn{\tb{Ban}}_n$ with the Banach-Mazur distance $d_{BM}$. It is given for~$X, Y \in \tn{\tb{Ban}}_n$ by
\begin{equation}
d_{BM}(X,Y):=\inf \left\{\log \left(||T||\cdot||T^{-1}||\right)\ |\ T:X\rightarrow Y \tn{ linear isomorphism}\right\}.
\end{equation}
Endowed with the Banach-Mazur distance $\tn{\tb{Ban}}_n$ becomes a compact connected (semi)metric space, see for example~\cite{Tho96}.
\begin{Lem}
$C^\mathcal{A}(\cdot):\tn{\tb{Ban}}_n\rightarrow \R$ is continuous.
\end{Lem}
\begin{proof}
Let $T:X\rightarrow Y$ be such that $\log(||T||\cdot ||T\inv||)< \epsilon$. Then for every $\gamma:\K\rightarrow X$ and $f:\Dc \rightarrow X$ Lipschitz one has
\begin{equation}
\label{eq24}
\mathrm{e^{-\epsilon}}\cdot\ell(\gamma)\leq \ell(T \circ \gamma)\leq \e^{\epsilon}\cdot\ell ( \gamma)\ \ \& \ \ \e^{-2\epsilon}\cdot \A(f)\leq  \A(T\circ f)\leq \e^{2\epsilon}\cdot \A(f).
\end{equation}
From \eqref{eq24} it follows that
\begin{equation}
\left|\log (C^\mathcal{A}(X))-\log(C^\mathcal{A}(Y))\right|< 8 \epsilon.
\end{equation}
So $\log (C^\mathcal{A}(\cdot))$ is continuous on $\tn{\tb{Ban}}_n$ and hence so is $C^\mathcal{A}(\cdot)$.
\end{proof}
At this point we prove the following variant of Theorem~\ref{thm1.4} which holds without assumptions on the area functional~$\mathcal{A}$.
\begin{Thm}
\label{thm4.6}
$\tn{QIS}^\mathcal{A}(\tn{\tb{Ban}}_n)$ is a compact interval~$[l_n^\mathcal{A},r_n^\mathcal{A}]$ where
\begin{equation}
\frac{1}{16}\leq l_n^\mathcal{A}\leq \frac{1}{4\pi}\leq r_n^\mathcal{A}<\frac{1}{2\pi}
\end{equation}
 and $r_n^\mathcal{A}$ is nondecreasing in~$n$.
\end{Thm}
\begin{proof}
$\tn{QIS}^{\mathcal{A}}(\tn{\tb{Ban}}_n)$ is the image of the compact connected space $\tn{\tb{Ban}}_n$ under the continuous map $C^\mathcal{A}(\cdot)$ and hence a compact interval $[l^\mathcal{A}_n,r^\mathcal{A}_n]$. By Example~\ref{rem4.3} we have $l_n^\mathcal{A}\leq \frac{1}{4\pi}\leq r_n^\mathcal{A}$ and by~\cite{HT79}
\begin{equation}
\label{eq25}
l_2^{ht}=r_2^{ht}=\frac{1}{4\pi}.
\end{equation}
Furthermore by quasiconvexity of~$\mathcal{A}$, see~\cite{BI02}, it follows that $l_n^{ht}=\frac{1}{4\pi}$. Hence \eqref{eq13} and the minimality of~$\mathcal{A}^{cr}$ imply $l_n^\mathcal{A}\geq \frac{1}{16}$.\par 
Fix $X\in \tn{\tb{Ban}}_n$ such that $C^\mathcal{A}(X)=r_n^\mathcal{A}$ and an $\mathcal{A}$-extremal curve~$\gamma$ within~$X$. By Lemma~\ref{lem1.6} $\gamma$ cannot be an isometric embedding and hence Theorem~\ref{thm1.2} implies
\begin{equation}
r_n^\mathcal{A}=C^\mathcal{A}(\gamma)<\frac{1}{2\pi}.
\end{equation}
To see that $r_n^\mathcal{A}$ is nondecreasing it suffices to note that $C^\mathcal{A}(X\times \R)\geq C^\mathcal{A}(X)$. This is true because~$X$ is a $1$-Lipschitz retrect of~$X\times \R$.
\end{proof}
We call $X\in \tn{\tb{Ban}}_n$ an \emph{$\mathcal{A}$-extremal space} if $C^\mathcal{A}(X)=r^\mathcal{A}_n$.
\begin{Ex}
\label{ex4.7}
By~\eqref{eq25} any~$X\in \tn{\tb{Ban}}_2$ is $\mathcal{A}^{ht}$~extremal. By comparison to $\mathcal{A}^{ht}$ and equations~\eqref{eq23},~\eqref{eq13} and~\eqref{eq14} the $\mathcal{A}$-quadratic isoperimetric spectra of~$\tn{\tb{Ban}}_2$ for~$\mathcal{A}=\mathcal{A}^{ht},\mathcal{A}^{b},\mathcal{A}^{m*},\mathcal{A}^{ir}$ are given as stated in Section~\ref{subsec12}. The up to isometry unique extremal space in all these situations is~$\R^2_\infty$. By~\eqref{eq13} we can also add
\begin{equation}
\tn{QIS}^{cr}(\tn{\tb{Ban}}_2)=\left[\frac{1}{16},\frac{1}{4\pi}\right]
\end{equation}
to the list where by~\eqref{eq12} the unique extremal space is the Euclidean plane.
\end{Ex}
For $n\geq 3$ the question which spaces are extremal remains completely open.
\subsection{Lower bounds}
\label{subsec44}
To complete the proof of Theorem~\ref{thm1.4} by Theorem~\ref{thm4.6} it suffices to show that~$r_n^{ht}$ converges to~$\frac{1}{2\pi}$ as~$n\rightarrow \infty$. More precisely we will prove~\eqref{eq8}.
Remember that every separable metric space $X$ admits an isometric embedding~$\iota$ into the space~$l^\infty$ of bounded sequences endowed with the supremum norm. If~$X$ is compact this Kuratowski embedding~$\iota:X\rightarrow \ell^\infty $ is given by choosing a countable dense subset~$\{x_1,x_2,x_3,...\}$ of $X$ and setting
\begin{equation}
\iota(x):=(d(x,x_1),d(x,x_2),d(x,x_3),....).
\end{equation}
Let $S_n:=\{\phi_1,...,\phi_n\}\subset \K$ be a cyclically ordered subset of equidistant points. The Kuratowski embedding gives an isometric embedding of~$S_n$ into~$\R^n_\infty$. Hence the proof of~\eqref{eq8} and in particular Theorem~\ref{thm1.4} is completed by the following Lemma.
\begin{Lem}
\label{lem4.8}
Let $X$ be a geodesic metric space, $m\geq 2$ and $e:S_m\rightarrow X$ an isometric embedding. Then
\begin{equation}
C^{ht}(X)\geq \left(1-\frac{4}{m}\right)\cdot \frac{1}{2\pi}.
\end{equation}
\end{Lem}
\begin{proof}
We may extend $e$ to a $1$-Lipschitz curve $\gamma:\K\rightarrow X$ by defining $\gamma$ to equal a geodesic connecting $e(\phi_i)$ to $e(\phi_{i+1})$ on $[\phi_i,\phi_{i+1}]$. Let $\iota:\K\rightarrow l^\infty$ be the Kuratowski embedding. As $l^\infty$ is an injective metric space there is a $1$-Lipschitz map $f:X \rightarrow l^\infty$ such that 
\begin{equation}
f(e(x_i))=\iota(x_i)
\end{equation}
for all $i=1,...,m$,~see for example~\cite{Lan13}. Then Lemma~\ref{lem2.7} implies
\begin{equation}
\tn{Fill}^{ht}(\gamma)\geq \tn{Fill}^{ht}(f\circ \gamma)\geq \tn{Fill}^{ht}(\iota)-\frac{1}{2\pi}\cdot m \cdot \left(2\cdot \frac{2\pi}{m}\right)^2\geq \left(1-\frac{4}{m}\right)\cdot 2\pi
\end{equation}
As $\ell(\gamma)=2\pi$ this implies the claim.
\end{proof}
\begin{Rem}
There are two observations that allow to push the lower bound on the constants $r_n^\mathcal{A}$ a bit further if one desires.
\begin{enumerate}
\item The Kuratowski embedding of $S_{2n}$ into $\R^{2n}_\infty$ carries more information than necessary. In fact one can forget about half of the coordinates and even obtain an isometric embedding of~$S_{2n}$ into~$\R^{n}_\infty$. This leads to 
\begin{equation}
\label{eq26}
r_n^{ht}\geq C^{ht}(\R^n_\infty)\geq \left(1-\frac{2}{n}\right)\cdot \frac{1}{2\pi}.
\end{equation}
\item If $X$ is a polyhedral normed space such as~$\R^n_\infty$ then by~\eqref{eq12} and \eqref{eq14} one has $C^\mathcal{A}(X)> C^{ht}(X)$ for $\mathcal{A}=\mathcal{A}^b,\mathcal{A}^{ir},\mathcal{A}^{m*}$. In particular for all these area functionals the inequality~\eqref{eq26} is even strict. Similarly one can obtain explicit upper bounds on the constants $C^\mathcal{A}(X)$ for $\mathcal{A}=\mathcal{A}^{b},\mathcal{A}^{ht},\mathcal{A}^{cr}$ and fixed polyhedral finite dimensional normed space $X$ by comparing to $\mathcal{A}^{ir}$ instead of $\mathcal{A}^{ht}$.
\end{enumerate}
\end{Rem}
\begin{Rem}
\label{rem4.10}
There is also an isometric embedding~$j$ of~$S_{n}$ into~$\R^n_1=L^1(S_n)$ which is given by \begin{equation}
(j(\phi))(\psi)=\begin{cases}
\frac{\pi}{n} &, \ (\phi,\psi,\tau(\phi)) \tn{ is cyclically ordered \& } \psi \neq\tau(\phi).\\
0&, \ \tn{else.}
\end{cases}
\end{equation}
In particular $C^{ht}(\R^n_1)\geq \left(1-\frac{4}{n}\right)\cdot \frac{1}{2\pi}$ and~$C^{ht}(\ell^1)=\frac{1}{2\pi}$.
\end{Rem}
We finish our paper with the proof Theorem~\ref{thm1.3}. More generally we show that 
\begin{equation}
\label{eqbanana2}
\tn{QIS}^\mathcal{A}(\tn{\tb{Ban}})=\Big\{0\Big\}\cup \Big[\frac{1}{4\pi},\frac{1}{2\pi}\Big]
\end{equation}
as soon as~$\mathcal{A}\geq \mathcal{A}^{ht}$.
\begin{proof}[Proof of~\eqref{eqbanana2}]
By~\eqref{eq25},~\cite{BI02} and~\cite{Cre20} one has
\[
\frac{1}{4\pi}\leq C^\mathcal{A}(X)\leq \frac{1}{2\pi}
\] for every nontrivial Banach space~$X$ and by Theorem~\ref{thm1.4} the interval $[\frac{1}{4\pi},\frac{1}{2\pi})$ is contained in $\tn{QIS}^\mathcal{A}(\tn{\tb{Ban}})$. Thus the proof is completed by Remark~\ref{rem4.10} and noting that $C^\mathcal{A}(\R)=0$.
\end{proof}
\bibliographystyle{alpha}
\bibliography{QIIinFinDim}
\end{document}